\documentclass[11pt,reqno,british,twoside]{article}

\usepackage[utf8]{inputenc}
\usepackage[sc]{mathpazo}
\usepackage{xspace,amssymb,amsmath,amsfonts,mathtools,url}
\usepackage[margin=30pt,font=small,labelfont=bf,labelsep=period]{caption}
\usepackage{array,paralist,fancyhdr}
\usepackage[numbers]{natbib}

\rmfamily
\usepackage[thisfountwidth,a4paper,wider,headers,verbose]{rmpage}

\newtheorem{theorem}{Theorem}[section]

\newtheorem{proposition}[theorem]{Proposition}
\newtheorem{corollary}[theorem]{Corollary}
\newtheorem{DEFINITION}[theorem]{Definition}
\newtheorem{REMARK}[theorem]{Remark}
\newtheorem{EXAMPLE}[theorem]{Example}

\newenvironment{definition}{\begin{DEFINITION}\normalfont}{\end{DEFINITION}}

\newenvironment{example}{\begin{EXAMPLE}\normalfont}{\end{EXAMPLE}}
\newenvironment{proof}{\noindent\textbf{Proof }}{\hfill $\Box$\medskip}

\numberwithin{equation}{section}
\numberwithin{table}{section}


\newcolumntype{C}{>{$}c<{$}}
\newcolumntype{L}{>{$}l<{$}}
\newcolumntype{R}{>{$}r<{$}}

\newcommand{\tablehead}[4]{\multicolumn{#1}{#2}{\textbf{#4}}}

\newcommand{\bC}{\mathbb C}
\newcommand{\bH}{\mathbb H}

\newcommand{\bR}{\mathbb R}
\newcommand{\bZ}{\mathbb Z}

\newcommand{\HKT}{\textsc{hkt}\xspace}
\newcommand{\SHKT}{\textsc{shkt}\xspace}
\newcommand{\NK}{\textsc{nk}\xspace}

\newcommand{\Lie}[1]{\operatorname{\textsl{#1}}}
\newcommand{\lie}[1]{\operatorname{\mathfrak{#1}}}
\newcommand{\Lel}[1]{{\mathsf{#1}}}  % Lie algebra element

\newcommand{\GL}{\Lie{GL}}

\newcommand{\sP}{\lie{sp}}
\newcommand{\Sp}{\Lie{Sp}}
\newcommand{\SU}{\Lie{SU}}
\newcommand{\su}{\lie{su}}

\newcommand{\un}{\lie{u}}
\newcommand{\g}{\lie{g}}
\newcommand{\h}{\lie{h}}
\newcommand{\lk}{\lie{k}}
\newcommand{\lr}{\lie{r}}
\newcommand{\ld}{\lie{d}}
\newcommand{\lf}{\lie{f}}
\newcommand{\lp}{\lie{p}}

\newcommand{\hook}{\mathop{\lrcorner}}

\newcommand{\Pg}[1][\g]{\mathcal P_{#1}}
\newcommand{\dP}{d_{\mathcal P}}

\DeclareMathOperator{\CP}{\bC P}
\DeclareMathOperator{\diag}{diag}

\DeclareMathOperator{\trace}{tr}
\DeclareMathOperator{\ad}{ad}
\DeclareMathOperator{\Ad}{Ad}
\DeclareMathOperator{\stab}{stab}

\DeclarePairedDelimiter{\abs}{\lvert}{\rvert}
\DeclarePairedDelimiter{\Span}{\langle}{\rangle}
\DeclarePairedDelimiter{\inpd}{\langle}{\rangle}
\newcommand{\inp}[2]{\inpd{#1,#2}}

\newcommand{\eqbreak}{\\&\qquad}

\newenvironment{smallpmatrix}{\left(\begin{smallmatrix}}{\end{smallmatrix}\right)}

\begin{document}
\thispagestyle{empty}
\begin{small}
  \begin{flushright}
    IMADA Preprint 2010\\
    CP\textsuperscript3-ORIGINS: 2010-52
  \end{flushright}
\end{small}

\begin{center}
  \LARGE\bfseries Homogeneous spaces, multi-moment maps and
  (2,3)-trivial algebras
\end{center}

\begin{center}
  \Large Thomas Bruun Madsen and Andrew Swann
\end{center}

\begin{abstract}
  For geometries with a closed three-form we briefly overview the
  notion of multi-moment maps.  We then give concrete examples of
  multi-moment maps for homogeneous hypercomplex and nearly Kähler
  manifolds.  A special role in the theory is played by Lie algebras
  with second and third Betti numbers equal to zero.  These we call
  (2,3)-trivial.  We provide a number of examples of such algebras
  including a complete list in dimensions up to and including five.
\end{abstract}

\section{Introduction}

Symplectic geometry is a geometrisation of the theory of mechanical
systems.  A symplectic structure on a manifold \( M \) is defined by a
closed \( 2 \)-form \( \omega \) that may be expressed locally as \(
\omega = dq_1 \wedge dp_1 + \dots + dq_n \wedge dp_n \), where \( \dim
M = 2n \).  In this context the concepts of linear and angular
momentum are captured by the notion of moment map corresponding to a
group of symmetries \( G \) of \( M \) preserving \( \omega \).  It is
an equivariant map \( \mu\colon M\to \lie g^* \), to the dual of the
Lie algebra \( \g \) of \( G \), characterised by the equation \(
d\inp \mu{\Lel X} = X \hook \omega \), for each \( \Lel X \in \lie g
\).  Here \( X \) denotes the corresponding vector field on \( M \)
generated by \( \Lel X \in \lie g \).

Developments in string and other field theories with Wess-Zumino terms
\cite{Michelson-S:conformal,Strominger:superstrings,Gates-HR:twisted,Baez-HR:string}
have highlighted the importance of geometries associated to closed \(
3 \)-forms \( c \).  Geometric aspects of these theories are not so
well developed.  In \cite{Madsen-S:multi-moment}, developing
\cite{Madsen:torsion}, we introduce a notion of multi-moment map
adapted to such geometries.  The important features of our definition
are that the target space depends only on the symmetry group \( G \)
and not the underlying manifold \( M \), in contrast to
\cite{Carinena-CI:multisymplectic}, and that existence of such maps
are guaranteed in many circumstances.  After reviewing the definition
and basic properties, the rest of the paper is devoted to giving
examples from hypercomplex and nearly Kähler geometry and to
classifications of a special class of Lie algebras that arises.

\medbreak We say the \( 3 \)-form \( c \in \Omega^3(M) \) defines a
\emph{strong geometry} if \( c \)~is closed, meaning \( dc = 0 \).
The lack of simple canonical descriptions for \( 3 \)-forms means that
in general a non-degeneracy assumption is not appropriate.  However,
if \( X\hook c = 0 \) occurs only for \( X = 0 \in T_xM \), we will
say that the structure is \emph{2-plectic}, following
\cite{Baez-HR:string}.

Suppose \( G \) acts on a strong geometry \( (M,c) \) preserving \( c
\).  Then for each \( \Lel X \in \g \), we have from the Cartan
formula that \( X \hook c \) is a closed \( 2 \)-form.  Similarly, if
\( \Lel X \) and \( \Lel Y \) are two commuting elements of \( \g \),
then \( (X\wedge Y) \hook c = c(X,Y,\cdot) \) is a closed \( 1
\)-form.  If this can be integrated to a function \( \nu_{X\wedge Y}
\), then we have a multi-moment map for the \( \bR^2 \) action
generated by \( X \) and \( Y \).

In general, the space of commuting elements in \( \g \) forms a
complicated variety.  We therefore introduce the \emph{Lie kernel} \(
\Pg \) as the kernel of the map \( \Lambda^2\lie g \to \lie g \)
induced by the Lie bracket \( [\cdot,\cdot] \).  A calculation shows
that for \( \Lel p = \sum_i X_i \wedge Y_i \in \Pg \), the one-form \(
p\hook c \) is closed.  This leads to

\begin{definition} \cite{Madsen-S:multi-moment} A \emph{multi-moment
    map} for the action of a group \( G \) of symmetries of a strong
  geometry \( (M,c) \) is an equivariant map \( \nu\colon M \to \Pg^*
  \) satisfying \( d\inp\nu{\Lel p} = p\hook c \), for all \( \Lel p
  \in \Pg \).
\end{definition}

The following result summarises the important existence results for
multi-moment maps.

\begin{theorem}\textup{\cite{Madsen-S:multi-moment}} Multi-moment maps
  exist for the action of \( G \) on \( (M,c) \) if either
  \begin{asparaenum}[(a)]
  \item\label{item:M-compact} \( M \) is compact with \( b_1(M) = 0
    \),
  \item\label{item:G-compact} \( G \) is compact, \( b_1(M) = 0 \) and
    \( G \) preserves a volume form on \( M \),
  \item\label{item:G-exact} \( c = d b \) with \( b \in \Omega^2(M) \)
    preserved by \( G \), or
  \item\label{item:2-3-trivial} \( b_2(\g) = 0 = b_3(\g) \).
  \end{asparaenum}
\end{theorem}

In this paper, we will give examples related to cases
(\ref{item:M-compact}--\ref{item:G-exact}) in hypercomplex and nearly
Kähler geometries. Underlying the nearly Kähler examples is an
analogue of the Kostant-Kirillov-Souriau construction, which we prove
in \cite{Madsen-S:multi-moment}. This construction is also relevant
for the Lie algebras satisfying the conditions of case
(\ref{item:2-3-trivial}); these we have dubbed as
\emph{(2,3)-trivial}.  This new class replaces that of semi-simple
algebras in the theory of symplectic moment maps, since semi-simple
algebras \( \h \) are characterised by \( b_1(\h) = 0 = b_2(\h) \).
We have found the following structure theorem:

\begin{theorem}\label{thm:2-3-trivial}
  \textup{\cite{Madsen-S:multi-moment}} A Lie algebra \( \g \) is \(
  (2,3) \)-trivial if and only if \( \g \) is solvable, with derived
  algebra \( \lk = \g' \) of codimension one and satisfying \(
  H^1(\lk)^{\g} = 0 = H^2(\lk)^{\g} = H^3(\lk)^{\g} \).
\end{theorem}

In this paper we use this result to give full lists of \( (2,3)
\)-trivial algebras in dimensions at most five and to show that every
nilpotent algebra \( \lk \) of dimension at most six occurs in the
above theorem.  We also give some constructions of infinite families
of \( (2,3) \)-trivial Lie algebras.

\section{A hypercomplex example}
\label{sec:exhc}

An \emph{\HKT structure} on a manifold is given by three complex
structures \( I \), \( J \), \( K = IJ = -JI \) with common Hermitian
metric such that \( Id\omega_I = Jd\omega_J = Kd\omega_K \).  By
\cite[Prop. 6.2]{Cabrera-S:aqH-torsion} it is unnecessary to check
integrability of \( I \), \( J \) and \( K \).  An example of a
homogeneous \HKT manifold is the compact simple Lie group \( \SU(3)
\).  In fact it admits a left-invariant \SHKT structure, meaning that
\( c = -Id\omega_I \) is closed.  As a consequence of left-invariance,
we may think of the \HKT structure on \( \SU(3) \) in terms of a
corresponding algebraic structure on its Lie algebra.  Write \( E_{pq}
\) for the elementary \( 3\times 3 \)-matrix with a \( 1 \) at
position \( (p,q) \).  Then the following \( 8 \) complex matrices
constitute a basis for \( \su(3) \): \( A_1=i(E_{11}{-}E_{22}) \), \(
A_2=i(E_{22}{-}E_{33}) \), \( B_{pq}=E_{pq}{-}E_{qp} \), \(
C_{pq}=i(E_{pq}{+}E_{qp}) \), for \( p=1,2<q=2,3 \).  We write \(
a_1,\dots,c_{23} \) for the dual basis.  Using the formula \(
d\alpha(X,Y)=-\alpha([X,Y]) \), one finds that
\begin{equation}
  \label{eq:extdi}
  \begin{gathered}
    da_1 = - 2b_{12}c_{12} - 2b_{13}c_{13},\
    da_2 = - 2b_{13}c_{13} - 2b_{23}c_{23},\\
    db_{12} = (2a_1 - a_2)c_{12} + b_{13}b_{23} + c_{13}c_{23},\
    db_{13} = (a_1 + a_2)c_{13} - b_{12}b_{23} + c_{12}c_{23},\\
    db_{23} = (-a_1 + 2a_2)c_{23} + b_{12}b_{13} + c_{12}c_{13},\
    dc_{12} = (-2a_1 + a_2)b_{12} - b_{13}c_{23} - b_{23}c_{13},\\
    dc_{13} = (-a_1 - a_2)b_{13} - b_{12}c_{23} + b_{23}c_{12},\
    dc_{23} = (a_1 - 2a_2)b_{23} + b_{12}c_{13} + b_{13}c_{12},
  \end{gathered}
\end{equation}
where \( \wedge \) signs have been omitted.

A metric is provided by minus the Killing form on \( \su(3) \); here
we may take \( g \) to be the map \( (X,Y)\mapsto-\trace(XY) \): \( g
= 2a_1^2 - a_1a_2 + 2( a_2^2 + b_{12}^2 + b_{13}^2 + b_{23}^2 +
c_{12}^2 + c_{13}^2 + c_{23}^2) \).

In \cite[Thm.  4.2]{Joyce:hypercomplex} Joyce proved the existence of
hypercomplex structures on certain compact Lie groups.  For \( \SU(3)
\), Joyce's hypercomplex structure comes from taking a highest root \(
\su(2) \), e.g., the span of \( A_1 \), \( B_{12} \), \( C_{12} \) and
the complement \( \bH+\bR \), \( \bH \cong \Span{ B_{13}, C_{13},
  B_{23}, C_{23} } \) and \( \bR \cong \Span{A_1{+}2A_2} \).
Concretely, let us write \( \mathbf I=A_1 \), \( \mathbf J=B_{12} \)
and \( \mathbf K=C_{12} \).  Then define \( I \) on \( \bH \) to be \(
\ad_{\mathbf I} \).  Similarly \( J \) and \( K \) act on \( \bH \) by
\( \ad_{\mathbf J} \) and \( \ad_{\mathbf K} \), respectively.  On \(
\bR + \su(2)\) the actions of \( I \), \( J \) and \( K \) are given
by \( IV = \mathbf I \), \( JV = \mathbf J \) and \( KV=\mathbf K \),
where \( V=(A_1 + 2A_2)/\sqrt 3 \).  The complex structures \( I \),
\( J \) and \( K \) are now determined completely by the requirement
that they square to \( -1 \).  It is straightforward to check that \(
IJ = K = -JI \).  Computations show that \( I \), \( J \) and \( K \)
are compatible with the metric, meaning \( g(X,Y)=g(IX,IY) \), etc.,
for all \( X,Y\in\su(3) \).  Defining \( 2a_1'=2a_1- a_2 \) and \(
2a_2'= \sqrt3 a_2 \), we find that the non-degenerate \( 2 \)-forms \(
\omega_I = g(I\cdot,\cdot) \), \( \omega_J \) and \( \omega_K \) are
given by
\begin{gather*}
  \omega_I = -a'_1 a'_2 + b_{12} c_{12} + b_{13} c_{13} - b_{23}
  c_{23},\ \omega_J = a_2' b_{12} - a_1' c_{12} - b_{13}
  b_{23} - c_{13} c_{23},\\
  \omega_K = a_2' c_{12}+a_1' b_{12} + b_{13} c_{23} + b_{23} c_{13}.
\end{gather*}
Using \eqref{eq:extdi} it then follows that
\begin{align*}
  d\omega_I&= - \sqrt{3}a_1'(b_{13}c_{13} + b_{23}c_{23}) +
  a_2'(2b_{12}c_{12} + b_{13}c_{13} - b_{23}c_{23})\eqbreak -
  b_{12}b_{13}c_{23} - b_{12}b_{23}c_{13} - b_{13}b_{23}c_{12} -
  c_{12}c_{13}c_{23},\\
  d\omega_J&=2a_1'a_2'c_{12} + a_1'(b_{13}c_{23} + b_{23}c_{13}) -
  a_2'(b_{13}b_{23} + c_{13}c_{23})\eqbreak -
  \sqrt{3}b_{12}b_{13}c_{13} - \sqrt{3}b_{12}b_{23}c_{23} +
  b_{13}c_{12}c_{13} - b_{23}c_{12}c_{23},\\
  d\omega_K&= - 2a_1'a_2'b_{12} + a_1'(b_{13}b_{23} + b_{23}c_{13}) +
  a_2'(b_{13}c_{23} + b_{23}c_{13})\eqbreak +
  \sqrt{3}b_{13}c_{12}c_{13} + \sqrt{3}b_{23}c_{12}c_{23} +
  b_{12}b_{13}c_{13} - b_{12}b_{23}c_{23}.
\end{align*}
From these formulae and the actions of \( I \), \( J \) and \( K \),
we verify the \( \HKT \) condition:
\begin{equation*}
  \begin{split}
    Id\omega_I&=a_1(2b_{12}c_{12} + b_{13}c_{13} - b_{23}c_{23}) -
    a_2(b_{12}c_{12} - b_{13}c_{13} - 2b_{23}c_{23})\\
    &\quad-b_{23}c_{12}c_{13} - b_{13}c_{12}c_{23} -
    b_{12}c_{13}c_{23} - b_{12}b_{13}b_{23}\\
    &= Jd\omega_J = Kd\omega_K.
  \end{split}
\end{equation*}
It is easy to check that \( dc=0 \), and thus \( (\SU(3),g,I,J,K) \)
is indeed an \( \SHKT \) manifold as claimed.

Unfortunately the multi-moment map for the left action of \( \SU(3) \)
on the strong geometry \( (\SU(3),c) \) is trivial.  However, we may
instead turn our attention towards the multi-moment maps \( \nu_I \),
\( \nu_J \) and \( \nu_K \) associated with the \( 3 \)-forms \(
d\omega_I \), \( d\omega_J \) and \( d\omega_K \) on \( \SU(3) \).  As
an example let us consider the multi-moment map \( \nu_I\colon \SU(3)
\to \Pg[\su(3)]^* \), \( \inp{\nu}{\Lel{p}} = \omega_I(p) \).  The
commutation relations for the chosen \( \su(3) \)-basis may be used to
establish a basis for \( \Pg[\su(3)] \leqslant \Lambda^2\su(3) \)
whilst the exterior derivative, via equations \eqref{eq:extdi}, gives
a basis for the submodule \( \su(3)\leqslant\Lambda^2\su(3) \).  We
may now decompose \( \omega_I \) at the identity: \( \omega_I =
\omega_I^{\su(3)} + \omega_I^{\Pg[\relax]} = 2\bigl(b_{12}c_{12} +
b_{13}c_{13}\bigr) + \bigl( - {\sqrt{3}a_1a_2}/2 - (b_{12}c_{12} +
b_{23}c_{23} - b_{13}c_{13})\bigr) \).  It follows that
\begin{equation*}
  \Ad_{g^{-1}}^*\nu_I(g) = -\tfrac{\sqrt3}2a_1a_2 - (b_{12}c_{12} +
  b_{23}c_{23} - b_{13}c_{13}).
\end{equation*}
The image of \( \nu_I \) is the orbit of \( \beta_I = \nu_I(e) \)
under \( \SU(3) \).  This element is preserved by a maximal torus,
invariant under \( I \), and its orbit is a copy of \( F_{1,2}(\bC^3)
\) inside \( \Pg[\su(3)]^* \).  At the algebraic level this is easily
verified:
\begin{equation*}
  \begin{split}
    \ker (\nu_I)_*&= \ker d\nu_I = \bigl\{\, A\in\su(3) :
    d\nu_I(p,A)=0 \text{ for all }
    \Lel p \in\Pg[\su(3)] \,\bigr\} \\
    &= \bigl\{\, A\in\su(3) : c(Ip,IA)=0 \text{ for all }
    \Lel p \in \Pg[\su(3)] \,\bigr\}\\
    &= I\bigl\{\, A\in\su(3) : g([Ip],A)=0\text{ for all } \Lel p \in
    \Pg[\su(3)] \,\bigr\} = [I\Pg[\su(3)]]^\bot = \Span{A_1,V}.
  \end{split}
\end{equation*}
Similarly the images of \( \nu_J \) and \( \nu_K \) are full flags in
\( \bC^3 \).  We find
\begin{alignat*}{3}
  &\begin{aligned}[b] \Ad^*_{g^{-1}}\nu_J(g)&= \tfrac{\sqrt3}2
    \left(2(a_1+a_2)b_{12} -
      (b_{23}c_{13}+b_{13}c_{23})\right)\\
    &\quad
    +\tfrac1{14}\left(2(2a_1-a_2)c_{12}-5(b_{13}b_{23}+c_{13}c_{23})\right),
  \end{aligned}&&\ker(\nu_J)_*=\Span{V,B_{12}},
  \\
  &\begin{aligned}[b] \Ad^*_{g^{-1}}\nu_K(g)&= \tfrac{\sqrt3}{14}
    \left((3a_1+2a_2)c_{12}
      -2(b_{13}b_{23}+c_{13}c_{23})) \right)\\
    &\quad -\tfrac12
    \left(2(8a_1+5a_2)b_{12}-11(b_{13}c_{23}+b_{23}c_{13})\right),
  \end{aligned}
  &\quad&\ker(\nu_K)_*=\Span{V,C_{12}}.
\end{alignat*}
Putting these together, we get an equivariant map \( \underline{\nu} =
(\nu_I,\nu_J,\nu_K) \colon \SU(3) \to (\Pg[\su(3)]^*)^3 \). The image
is the Aloff-Wallach space \( A_{1,1}=\SU(3)/{T^1_{1,1}} \).  The
relatively high dimension of this image indicates that multi-moment
maps could be a useful tool to study homogeneous hyperHermitian
structures.
 
\section{Six-dimensional nearly Kähler manifolds}

A \emph{nearly Kähler structure}, briefly an \NK structure, on a
six-dimensional manifold may be specified by a \( 2 \)-form~\( \sigma
\) and a \( 3 \)-form \( \psi^+ \) whose common pointwise stabiliser
in \( \GL(6,\bR) \) is isomorphic to~\( \SU(3) \).  The \NK condition
is then \( d\sigma = \psi^+ \) and \( d\psi^- = -\tfrac23\sigma^2 \),
where \( \psi^++i\psi^- \in \Lambda^{3,0} \).  We will indicate how
each homogeneous strict \NK six-manifold \( G/H = F_{1,2}(\bC^3) \),
\( \CP(3) \), \( S^3\times S^3 \) and \( S^6 \), as classified by
\citet{Butruille:nK}, may be realised as a 2-plectic orbit \(
G\cdot\beta \) in~\( \Pg^* \).  Let \( \dP\colon \Pg^* \to
\Lambda^3\g^* \) be the map induced by \( d \).  Then our realisation
is such that \( \Psi = \dP\beta \) induces \( \psi^+ \) via \(
\inp{\Psi}{\Lel{X}\wedge\Lel{Y}\wedge\Lel{Z}}=\psi^+(X,Y,Z) \) and \(
\sigma(X,Y) = \beta(\Lel X,\Lel Y) \).

In each case the element \( \beta\in\Pg^* \) must be chosen with some
care.  For instance neither of the \( 3 \) copies \( F_{1,2}(\bC^3)
\subset \Pg[\su(3)]^* \) from section \ref{sec:exhc} behaves in the
manner described above.  On the other hand \( \SU(3) \) acting on the
element \( \beta_1 = b_{12}c_{12} + c_{13}b_{13} + b_{23}c_{23} \in
\Pg[\su(3)]^* \) gives a copy of the full flag with forms \(
\dP\beta_1 \) and \( \beta_1 \) inducing the \NK structure.  The
associated almost complex structure \( J \) is given by \( J(B_{12}) =
C_{12} \), \( J(C_{13}) = B_{13} \), \( J(B_{23}) = C_{23} \).

To obtain \( \CP(3) \), we let \( Sp(2) \) act on \( \beta_2=a_1b_{11}
+ b_{12}r + c_{12}q \) in \( \Pg[\sP(2)]^* \).  The chosen basis for
\( \sP(2) \) consists of the \( 10 \) complex matrices \( A_1 =
i(E_{11} - E_{33}) \), \( A_2 = i(E_{22} - E_{44}) \), \( Q = E_{12} -
E_{21} + E_{34} - E_{43} \), \( R = i(E_{12} + E_{21} - E_{34} -
E_{43}) \), \( B_{k\ell } = E_{k,2+\ell } + E_{\ell ,2+k} -
E_{2+k,\ell } - E_{2+\ell ,k} \), \( C_{k\ell } = i(E_{k,2+\ell } +
E_{\ell ,2+k} + E_{2+k,\ell } + E_{2+\ell ,k}) \), \( 1\leqslant
k\leqslant \ell \leqslant 2 \).  One easily verifies that \(
\stab_{\sP(2)}\beta_1 = \un(1)\oplus\su(2) \), so that, up to discrete
coverings, the orbit of \( \beta_2 \) is \( \CP(3) \). We have
\begin{equation*}
  \begin{gathered}
    da_1 = - 2(4b_{11}c_{11}+b_{12}c_{12}+qr),\
    db_{11} = 2a_1c_{11}+b_{12}q-c_{12}r,\\
    db_{12} = (a_1+a_2)c_{12} + 2(-b_{11}+b_{22})q -
    2(c_{11}+c_{22})r,\\
    dc_{12} = - (a_1 + a_2)b_{12} + 2(b_{11}+b_{22})r +
    2(-c_{11}+c_{22})q,\\
    dq = (a_1 - a_2)r + 2(b_{11}-b_{22})b_{12} +
    2(c_{11}-c_{22})c_{12},\\
    dr = (-a_1 + a_2)q + 2 (c_{11}+c_{22})b_{12} -
    2(b_{11}+b_{22})c_{12}.
  \end{gathered}
\end{equation*}
Computations now show that \( \beta_2 \) and \( \dP\beta_2 \)
determine a \NK structure which has almost complex structure given by
\( J(A_1) = \tfrac12 B_{11} \), \( J(B_{12}) = R \) and \( J(C_{12}) =
Q \).

The homogeneous \NK structure on \( S^3\times S^3 \) is obtained on
the orbit of \( \beta_3 = e_1f_1 + e_2f_2 + e_3f_3 \in
\Pg[\su(2)\oplus\su(2)]^* \).  Here \( e_i \), \( f_i \) denotes a
cyclic basis for \( \su(2)^*\oplus\su(2)^* \), meaning \( de_i =
e_{i+1}e_{i+2} \) and \( df_i = f_{i+1}f_{i+2} \) for \( i \in \bZ/3
\).  One may verify that the associated almost complex structure is
given by \( J(E_i) = (E_i+2F_i)/\sqrt3 \).

Finally we obtain \( S^6 \) as the \( G_2 \)-orbit of the element \(
\beta_4 = b_1c_1 + b_3c_3 + c_4b_4 \in \Pg[\g_2]^* \).  The chosen
basis for \( \g_2 \) is \( A_1 = iH_1 \), \( A_2 = iH_2 \), \( B_a =
X_a - Y_a \), \( C_a = i(X_a + Y_a) \), \( (1\leqslant a\leqslant 6)
\), with the elements \( \{H_1,\dots,Y_6\} \) defined in \cite[Table
22.1]{Fulton-Harris:rep}. Now we have
\begin{equation*}
  \begin{gathered}
    db_1 = (2a_1-a_2)c_1 + b_3b_2 + c_3c_2 + 2(b_4b_3 + c_4c_3) +
    b_4b_5 + c_4c_5,\\
    dc_1 = (-2a_1+a_2)b_1 + c_3b_2 + c_2b_3 + 2(c_4b_3 + c_3b_4) +
    b_4c_5 + b_5c_4,\\
    db_3 = (-a_1+a_2)c_3 + b_2b_1 + c_1c_2 + 2(b_1b_4 + c_1c_4) +
    b_4b_6 + c_4c_6,\\
    dc_3 = (a_1-a_2)b_3 + c_2b_1 + b_2c_1 + b_4c_6 + 2(b_1c_4 +b_4c_1)
    + b_4c_6 + b_6c_4,\\
    db_4 = a_1c_4 + 2(b_3b_1 + c_1c_3) + b_5b_1+c_5c_1 + c_6c_3 +
    b_6b_3,\\
    dc_4 = b_4a_1 + 2(b_1c_3 + b_3c_1) + c_5b_1 + c_1b_5 + c_6b_3 +
    c_3b_6,
  \end{gathered}
\end{equation*}
and it can be verified that \( \beta_4 \) and \( \dP\beta_4 \) induce
a \NK structure on \( S^6 \). In this case one has \( J(B_1) = C_1 \),
\( J(B_3) = C_3 \) and \( J(C_4) = B_4 \).

\medbreak Motivated by these examples, it makes sense to study the
multi-moment map formalism for \NK manifolds with less symmetry.

\begin{table}
  \centering
  \begin{tabular}{LCCR}  
    \hline
    G & \beta & \dP\beta & \mathcal{O}=G\cdot\beta \\
    \hline
    \SU(3) & b_{12}c_{12}{+}c_{13}b_{13}{+}b_{23}c_{23} &
    3(b_{12}(b_{13}c_{23}{+}b_{23}c_{13}){+}c_{12}(b_{13}b_{23}{+}c_{13}c_{23}))
    & F_{1,2}(\bC^3) \\ 
    \Sp(2) & a_1b_{11}{+}b_{12}r{+}c_{12}q &
    {-}3(a_1(b_{12}q-c_{12}r)+2b_{11}(b_{12}c_{12}+qr)) & \CP(3) \\ 
    \SU(2)^2 & e_1f_1{+}e_2f_2{+}e_3f_3 &
    e_{12}f_3+e_{23}f_1+e_{31}f_2-e_1f_{23}-e_2f_{31}-e_3f_{12} &
    S^3\times S^3 \\ 
    G_2 & b_1c_1{+}b_3c_3{+}c_4b_4 &
    6(b_1(b_3c_4{-}c_3b_4){-}c_1(b_3b_4+c_3c_4)) & S^6 \\ 
    \hline
  \end{tabular}
  \caption{Realisations of the six-dimensional nearly Kähler manifolds
    as orbits in Lie kernels.} 
  \label{tab:nK}
\end{table}

\section{Positive gradings of nilpotent algebras}
\label{sec:posgr}

A Lie algebra \( \lk \) is \emph{positively graded} if there is a
vector space direct sum decomposition \( \lk=\lk_1+\dots+\lk_r \) with
\( [\lk_i,\lk_j]\subseteq\lk_{i+j} \) for all \( i \), \( j \).  A
grading of an \( n \)-dimensional Lie algebra \( \lk \) may be
specified in terms of an \( n \) positive integers, see Example
\ref{ex:posgrad}. We have

\begin{proposition}
  Any nilpotent Lie algebra of dimension at most six admits a positive
  grading.  From dimension seven and above, there are nilpotent Lie
  algebras, which do not admit a positive grading.
\end{proposition}

The nilpotent Lie algebras of dimension at most six and primitive
positive gradings are listed in Table \ref{tab:posgrad}. Example
\ref{ex:posgrad} illustrates how gradings are found. In
\cite{Madsen-S:multi-moment} we show that there are examples of
nilpotent algebras \( \lk \) of dimension seven that can not arise as
the derived algebra of any \( (2,3) \)-trivial algebra \( \g \).  It
follows that such \( \lk \) can not admit a positive grading.

Positive gradings can be used to generate \( (2,3) \)-trivial
algebras:

\begin{corollary}
  \label{cor:extposgr}
  Each of the \( 50 \) Lie algebras listed in Table \ref{tab:posgrad}
  is the derived algebra of a completely solvable \( (2,3) \)-trivial
  Lie algebra.
\end{corollary}

\begin{proof}
  Let \( \g=\Span{A}+\lk \), where \( \lk \) is one of the algebras of
  Table \ref{tab:posgrad} and \( \ad_A \) acts as multiplication by \(
  i \) on \( \lk_i \).  Then \( \g \) is a solvable algebra.  Moreover
  \( (\Lambda^s\lk)^{\g}=\{0\} \) for \( s\geqslant1 \), so that \(
  H^1(\lk)^{\g}=\{0\}=H^2(\lk)^{\g}=H^3(\lk)^{\g} \).  Thus, by
  Theorem~\ref{thm:2-3-trivial}, \( \g \) is \( (2,3) \)-trivial.
  Since \( \ad_X \) has real eigenvalues for each \( X\in\g \) the Lie
  algebra is completely solvable.
\end{proof}

\begin{example}
  \label{ex:posgrad}
  A Lie algebra \( \lk \) may be specified in terms of the action of
  \( d\colon \lk^*\to \Lambda^2\lk^* \) on a basis for \( \lk^* \).
  By \( (0^2,12,13,14{+}23,24{+}15) \) we thus denote the nilpotent
  Lie algebra \( \lk \) which has a basis \( e^1,\dots,e^6 \) for \(
  \lk^* \) satisfying \( de^1=0=de^2 \), \( de^3=e^1e^2 \), \dots, \(
  de^6 = e^2e^4 + e^1e^5 \).  Assigning weights can now be formulated
  schematically as follows: \( e^1\to a \), \( e^2\to b \), \( e^3\to
  a+b \), \( e^4\to2a+b \), \( e^5\to 3a+b=a+2b \), \(
  e^6\to2(a+b)=2(a+b) \).  In particular we see that \( 2a=b \), so
  that a grading may be defined by \( \lk=\lk_1\oplus\dots\oplus\lk_6
  \), where \( \lk_i=\Span{e_i} \).  Choosing \( a=1 \), this weight
  decomposition is denoted by \( 123456 \).  Following the proof of
  Corollary \ref{cor:extposgr}, we may now define a \( (2,3)
  \)-trivial extension of \( \lk \): \(
  (0,12,2.13,3.14{+}23,4.15{+}24,5.16{+}25{+}34,6.17{+}24{+}26) \).
\end{example}

\begin{table}
  \centering
  \begin{tabular}[t]{L@{\hspace{-1em}}R}
    \tablehead{1}{l}{b}{Structure} & \tablehead{1}{r}{b}{Grading} \\   
    \hline
    (0)                    & 1       \\
    \hline
    (0^2)                  & 1^2     \\
    \hline
    (0^3)                  & 1^3     \\
    (0^2,12)               & 1^22    \\ 
    \hline
    (0^4)                  & 1^4     \\
    (0^3,12)               & 1^32    \\
    (0^2,12,13)            & 1^223   \\ 
    \hline
    (0^5)                  & 1^5     \\
    (0^4,12),(0^4,12{+}34) & 1^42    \\
    (0^3,12,13)            & 1^32^2  \\
    (0^3,12,14)            & 1^323   \\
    (0^3,12,13{+}24)       & 1^22^23 \\
    (0^2,12,13,23)         & 1^223^2 \\    
    (0^2,12,13,14)         & 1^2234  \\
    (0^2,12,13,14{+}23)    & 12345   \\
    \hline
    (0^6)                  & 1^6     \\
    (0^5,12),(0^5,12+34)   & 1^52    \\[3ex]
    \hline
  \end{tabular}
  \quad
  \begin{tabular}[t]{L@{\hspace{-1em}}R}
    \tablehead{1}{l}{b}{Structure} & \tablehead{1}{r}{b}{Grading} \\   
    \hline
    (0^4,12,13),(0^4,13{+}42,14{+}23),            &           \\
    \quad (0^4,12,34),(0^4,12,14{+}23)            & 1^42^2    \\
    (0^4,12,15)                                   & 1^423     \\
    (0^3,12,13,23)                                & 1^32^3    \\
    (0^4,12,14{+}25),(0^4,12,15{+}34),            &           \\
    \quad (0^3,12,13,14),(0^3,12,23,14{\pm}35),   &           \\
    \quad (0^3,12,13,24),(0^3,12,13,14{+}35)      & 1^32^23   \\
    (0^3,12,14,24)                                & 1^323^2   \\    
    (0^3,12,14,15)                                & 1^3234    \\
    (0^3,12,13{+}14,24),(0^3,12,13{+}42,14{+}23)  &           \\
    \quad (0^3,12,13,14{+}23),(0^3,12,14,13{+}42) & 1^22^23^2 \\
    (0^3,12,14{-}23,15{+}34)                      & 1^22^234  \\
    (0^2,12,13,23,14{\pm}25),(0^2,12,13,23,14)    & 1^223^24  \\
    (0^2,12,13,14,15),(0^2,12,13,14,34{+}52)      & 1^22345   \\
    (0^3,12,14,15{+}23)                           & 1^23234   \\
    (0^3,12,14,15{+}24)                           & 121345    \\
    (0^3,12,14,15{+}23{+}24)                      & 123^245   \\
    (0^2,12,13,14+23,24+15)                       & 123456    \\
    (0^2,12,13,14{+}23,34{+}52)                   & 123457    \\
    (0^2,12,13,14,23{+}15)                        & 134567    \\ 
    \hline
  \end{tabular}
  \caption{Positive gradings of nilpotent Lie algebras of dimension \(
    \leqslant 6 \). Algebras are ordered according to their dimension
    and a primitive positive grading.} 
  \label{tab:posgrad}
\end{table}

\section{Families of (2,3)-trivial algebras}
\label{sec:clf23}

While the method of positive gradings provides an effective tool for
constructing examples of \((2,3)\)-trivial algebras, it is inadequate
if one aims for a general understanding of the \((2,3)\)-trivial
class. Therefore we now turn to give a complete list of such algebras
in dimensions up to and including five.

In dimension one, the only Lie algebra is Abelian and is automatically
\( (2,3) \)-trivial.  In dimension two a Lie algebra is either Abelian
or isomorphic to the \( (2,3) \)-trivial algebra \( (0,21) \).  These
first two examples are uninteresting from the point of view of
multi-moment maps since they have \( \Pg=\{0\} \).  In next dimensions
we have:

\begin{proposition}
  \label{prop:solb345}
  The \( (2,3) \)-trivial Lie algebras in dimensions three, four and
  five are listed in the Tables \ref{tab:34sp} and \ref{tab:5sp}.
\end{proposition}

We shall now give a proof of Proposition \ref{prop:solb345}. Note that
we do not discuss inequivalence of the algebras; imposing
inequivalence would put further restrictions on the parameters, see
for instance \cite[Theorem 1.1, 1.5]{Andrada-BDO:four}.

\begin{table}
  \centering
  \begin{tabular}{LCR}  
    \hline
    \lr_3 & (0,21{+}31,31) & \\
    \lr_{3,\lambda} & (0,21,\lambda.31) & \lambda \neq {-}1,0 \\
    \lr'_{3,\lambda} & (0,\lambda.21 {+}31,{-}21{+}\lambda.31) &
    \lambda\neq0 \\ 
    \lr_4 & (0,21{+}31,31{+}41,41) & \\
    \lr_{4,\lambda} & (0,21,\lambda.31{+}41,\lambda.41) & \lambda \ne
    {-}1,{-}\tfrac12,0 \\ 
    \lr_{4,\lambda(2)} & (0,21,\lambda_1.31,\lambda_2.41) &
    \lambda_i,\lambda_1{+}\lambda_2\neq{-}1,0 \\ 
    \lr'_{4,\lambda(2)} &
    (0,\lambda_1.21,\lambda_2.31{+}41,-31{+}\lambda_2.41) &
    \lambda_1\neq0,\lambda_2 \ne {-}\tfrac{\lambda_1}{2},0 \\ 
    \ld_{4,\lambda} & (0,21,\lambda.31,(1{+}\lambda).41{+}32) &
    \lambda\ne {-}2,{-}1,{-}\tfrac12,0 \\ 
    \ld'_{4,\lambda} &
    (0,\lambda.21{+}31,{-}21{+}\lambda.31,2\lambda.41{+}32) &
    \lambda\neq 0 \\ 
    \h_4 & (0,21{+}31,31,2.41{+}32) & \\
    \hline
  \end{tabular}
  \caption{The three- and four-dimensional \( (2,3) \)-trivial Lie
    algebras.  Our labelling of \( \ld_{\lambda} \)
    differs from \cite[Theorem  1.5]{Andrada-BDO:four}.} 
  \label{tab:34sp}
\end{table} 

\begin{table}
  \centering
  \begin{tabular}{LLL}  
    \hline
    \lr_5 & (0,21{+}31,31{+}41,41{+}51,51) & \\
    \lr_{5(1),\lambda} &
    (0,21,\lambda.31{+}41,\lambda.41{+}51,\lambda.51) &
    \lambda\ne{-}1,{-}\tfrac12,0 \\ 
    \lr_{5(2),\lambda} & (0,21{+}31,31,\lambda.41{+}51,\lambda.51) &
    \lambda\ne-2,-1,-\tfrac12,0 \\ 
    \lr_{5,\lambda(2)} &
    (0,21,\lambda_1.31,\lambda_2.41{+}51,\lambda_2.51) &
    \lambda_i\ne-1,0 ;   \lambda_1{+}\lambda_2\ne0,{-}1  ; \\ 
    &  & 1{+}2\lambda_2,\lambda_1{+}2\lambda_2\ne0 \\
    \lr_{5,\lambda(3)} & 
    (0,21,\lambda_1.31,\lambda_2.41,\lambda_3.51) & 
    \lambda_i\neq -1,0 ;   \lambda_1{+}\lambda_2{+}\lambda_3\ne0; \\
    && \lambda_i{+}\lambda_j\neq{-}1,0\ (i\ne j) \\
    \lr'_{5,\lambda(2)} &
    (0,\lambda_1.21{+}31,\lambda_1.31,\lambda_2.41{+}51,-41{+}\lambda_2.51)
    & \lambda_i,\lambda_1{+}2\lambda_2\ne0 \\ 
    \lr'_{5,\lambda(3)} &
    (0,\lambda_1.21,\lambda_2.31,\lambda_3.41{+}51,-41{+}\lambda_3.51)
    & \lambda_i{\ne} 0 ; \lambda_1{\ne} {-}\lambda_2;
    \lambda_1,\lambda_2{\ne} {-} 2\lambda_3\\
    \lr''_{5,\lambda} &
    (0,\lambda.21{+}31{+}41,{-}21{+}\lambda.31{+}51,
    \lambda.41{+}51,{-}41{+}\lambda.51) 
    & \lambda\ne0 \\ 
    \lr''_{5,\lambda(3)} &
    (0,\lambda_1.21{+}31,{-}21{+}\lambda_1.31,
    \lambda_2.41{+}\lambda_3.51, {-}\lambda_3.41{+}\lambda_2.51)
    & \lambda_i\ne0 \\ 
    \ld_{5(1)} & (0,21,21{+}31,31{+}41,2.51{+}32) & \\ 
    \ld^{\pm}_{5(2)} & (0,21,21{+}31,2.41,2.51\pm41{+}32) & \\
    \ld_{5(1),\lambda} &
    (0,21,\lambda.31,(1{+}\lambda).41,(1{+}\lambda).51{+}32{+}41) &
    \lambda\ne{-}2,{-}\tfrac32,{-}1,{-}\tfrac23,{-}\tfrac12,0 \\ 
    \ld_{5(2),\lambda} & (0,21,21{+}31,\lambda.41,2.51{+}32) &
    \lambda\ne-3,-1,0 \\ 
    \ld_{5,\lambda(2)} &
    (0,21,\lambda_1.31,\lambda_2.41,(1{+}\lambda_1).51{+}32) &
    \lambda_1\ne{-}2,{-}\tfrac12,{-}1,0 ;  \lambda_2\ne0,{-}1 ; \\ 
    &  & \lambda_1{+}\lambda_2\ne{-}2,0 ;  \lambda_2{+}2\lambda_1\ne{-}1 \\
    \ld_{5(3),\lambda} &
    (0,\lambda.21,31,31{+}41,(1{+}\lambda).51{+}32) &
    \lambda\ne{-}3,{-}2,{-}1,{-}\tfrac12,0 \\ 
    \ld'^{\pm}_{5,\lambda} &
    (0,\lambda.21{+}31,-21{+}\lambda.31,2\lambda.41,2\lambda.51\pm41{+}32)
    & \lambda\ne0 \\ 
    \ld'_{5,\lambda(2)} &
    (0,\lambda_1.21{+}31,-21{+}\lambda_1.31,\lambda_2.41,2\lambda_1.51)
    & \lambda_1,\lambda_2\ne0 \\ 
    \lp_5 & (0,21,21{+}31,2.41{+}32,3.51{+}42) & \\ 
    \lp_{5,\lambda} &
    (0,21,\lambda.31,(1{+}\lambda).41{+}32,(2{+}\lambda).51{+}42) &
    \lambda\ne {-}3,{-}2,{-}1,{-}\tfrac12,0 \\ 
    \hline
  \end{tabular}
  \caption{The five-dimensional \( (2,3) \)-trivial Lie
    algebras.}
  \label{tab:5sp}
\end{table}

The starting point for our analysis is Theorem~\ref{thm:2-3-trivial}
which gives \( \g=\bR A+\lk \), where \( \lk=\g' \) is nilpotent. The
element \( A \) acts on \( H^i(\lk) \) as endomorphism with
determinant~\( a_i \).  Now \( (2,3) \)-triviality of \( \g \) may be
rephrased as the non-vanishing of \( a_1 \), \( a_2 \) and \( a_3 \).

\paragraph{Dimension three}

Let \( \g \) be a \( (2,3) \)-trivial algebra of dimension three.
Then \( \lk \) is nilpotent and two-dimensional, so \( \lk\cong\bR^2
\).  The element \( A \) acts on \( \bR^2 \) invertibly and the
induced action on \( H^2(\bR^2) \cong \Lambda^2\bR^2 \cong \bR \) is
also invertible.  So either \( A \) is diagonalisable over \( \bC \)
with non-zero eigenvalues whose sum is non-zero, giving cases \(
\lr_{3,\lambda\neq-1,0} \) and \( \lr'_{3,\lambda\neq0} \), or \( A \)
acts with Jordan normal form \( \begin{smallpmatrix}
  \lambda&1\\0&\lambda \end{smallpmatrix} \), \( \lambda\ne0 \),
giving case \( \lr_3 \).

\paragraph{Dimension four}

For \( \g \) of dimension four we have \( \lk \cong \bR^3 \) or the
Heisenberg algebra \( \h_3 = (0^2,21) \).  The former gives the
algebras from the \( \lr \)- and \( \lr' \)-series.  Derivations of \(
\bR^3 \) are just linear endomorphisms; therefore the relevant list of
extensions of \( \bR^3 \) may be obtained from considerations of
invertible \( 3\times 3 \) matrices in normal form: \( A_1 =
\begin{smallpmatrix}
  1&0&0\\0&\lambda_1&0\\0&0&\lambda_2
\end{smallpmatrix}
\), \( A_2 =
\begin{smallpmatrix}
  1&0&0\\0&\lambda&1\\0&0&\lambda
\end{smallpmatrix}
\), \( A_3 =
\begin{smallpmatrix}
  1&1&0\\0&1&1\\0&0&1
\end{smallpmatrix}
\) and \( A_4 =
\begin{smallpmatrix}
  \lambda_1&0&0\\0&\lambda_2&1\\0&-1&\lambda_2
\end{smallpmatrix}
\).  The element \( A_1 \) gives the family \( \lr_{4,\lambda(2)} \),
and the induced action on \( H^1(\bR^3)\cong\bR^3 \) is, up to sign,
multiplication by the transpose of \( A_1 \).  Using this observation
one easily finds the induced actions on \( H^2(\bR^3)\cong\bR^3 \) and
\( H^3(\bR^3)\cong\bR \).  We deduce that \( (2,3) \)-triviality holds
if and only if the determinants \( a_1 = \lambda_1\lambda_2 \), \( a_2
= (1 + \lambda_1)(1 + \lambda_2)(\lambda_1 + \lambda_2) \) and \( a_3
= 1 + \lambda_1 + \lambda_2 \) do not vanish.  The matrix \( A_2 \)
gives us the algebra \( \lr_{4,\lambda} \).  In this case \( a_1 =
\lambda^2 \), \( a_2 = 2\lambda(1 + \lambda)^2 \), \( a_3 = 1 +
2\lambda \), giving the restrictions on parameters in
Table~\ref{tab:34sp}.  The algebra \( \lr_4 \) corresponds to \( A_3
\).  Finally, \( A_4 \) occurs when the action has \( 2 \) complex
eigenvalues.  The corresponding family is \( \lr'_{4,\lambda(2)} \),
where \( \lambda_1,\lambda_2 \) are restricted by \( a_i\ne0 \) for \(
a_1 = \lambda_1(1 + \lambda_2^2) \), \( a_2 = 2\lambda_2(1 +
(\lambda_1 + \lambda_2)^2) \) and \( a_3 = \lambda_1 + 2\lambda_2 \).

The Heisenberg algebra \( \h_3 \) has \( H^1(\h_3) \cong
\Span{e^1,e^2} \), \( H^2(\h_3) \cong \Span{e^{13},e^{23}} \), \(
H^3(\h_3) \cong \Span{e^{123}} \).  The action of \( A \), being a
derivation, is represented by a matrix of the form
\( \begin{smallpmatrix}
  B&0\\\underline{b}&\trace{B} \end{smallpmatrix} \) with \( B \) a
real \( 2\times 2 \)-matrix and \( \underline{b} = (b_1,b_2)\in\bR^2
\).  To see this, write \( \ad_A(e_i) = \sum_kb^k_ie_k \) and consider
the relation \( \ad_A(e_3) = \ad_A[e_1,e_2] = [\ad_A(e_1),e_2] +
[e_1,\ad_A(e_2)] \).  After the transformation \( A\to A-b_2 e_1+b_1
e_2 \) we may assume \( \underline{b} = 0 \), so that the algebras are
distinguished by the normal form of \( B \).  The family \(
\ld_{4,\lambda} \) arises when \( B = \diag(1,\lambda) \).
Restrictions on \( \lambda \) now follow from the determinants \( a_1
= \lambda \), \( a_2 = (2+\lambda)(1+2\lambda) \) and \( a_3 =
2(1+\lambda) \) being non-zero.  If \( B = \begin{smallpmatrix}
  1&1\\0&1 \end{smallpmatrix} \) one has the algebra \( \h_4 \).
Finally the action may have complex eigenvalues so that \( B
= \begin{smallpmatrix} \lambda&1\\-1&\lambda \end{smallpmatrix} \),
corresponding to the family \( \ld'_{4,\lambda} \).  One finds \( a_1
= 1+\lambda^2 \), \( a_2 = 1+9\lambda^2 \) and \( a_3 = 4\lambda \),
so we must have \( \lambda\neq0 \).

\paragraph{Dimension five}

A five-dimensional \( (2,3) \)-trivial Lie algebra has \(
\lk\cong\bR^4 \), \( (0^3,21) \) or \( (0^2,21,31) \).  In the Abelian
case \( H^1(\bR^4)\cong\bR^4 \), \( H^2(\bR^4)\cong\bR^6 \), \(
H^3(\bR^4)\cong\bR^4 \).  The solvable extensions are found by taking
invertible matrices in the normal forms \( A_1,\dots,A_9 \):
\begin{gather*}
  \begin{smallpmatrix}
    1&0&0&0\\
    0&\lambda_1&0&0\\
    0&0&\lambda_2&0\\
    0&0&0&\lambda_3
  \end{smallpmatrix}
  ,\quad
  \begin{smallpmatrix}
    1&0&0&0\\
    0&\lambda_1&0&0\\
    0&0&\lambda_2&1\\
    0&0&0&\lambda_2
  \end{smallpmatrix}
  ,\quad
  \begin{smallpmatrix}
    1&1&0&0\\
    0&1&0&0\\
    0&0&\lambda&1\\
    0&0&0&\lambda
  \end{smallpmatrix}
  ,\quad
  \begin{smallpmatrix}
    1&0&0&0\\
    0&\lambda&1&0\\
    0&0&\lambda&1\\
    0&0&0&\lambda
  \end{smallpmatrix}
  ,\quad
  \begin{smallpmatrix}
    1&1&0&0\\
    0&1&1&0\\
    0&0&1&1\\
    0&0&0&1
  \end{smallpmatrix},\\
  \begin{smallpmatrix}
    \lambda_1&0&0&0\\
    0&\lambda_2&0&0\\
    0&0&\lambda_3&1\\
    0&0&-1&\lambda_3
  \end{smallpmatrix}
  ,\quad
  \begin{smallpmatrix}
    \lambda_1&1&0&0\\
    0&\lambda_1&0&0\\
    0&0&\lambda_2&1\\
    0&0&-1&\lambda_2
  \end{smallpmatrix}
  ,\quad
  \begin{smallpmatrix}
    \lambda_1&1&0&0\\
    -1&\lambda_1 &0&0\\
    0&0&\lambda_2&\lambda_3\\
    0&0&-\lambda_3&\lambda_2
  \end{smallpmatrix}
  ,\quad
  \begin{smallpmatrix}
    \lambda&1&1&0\\
    -1&\lambda&0&1\\
    0&0&\lambda&1\\
    0&0&-1&\lambda
  \end{smallpmatrix}.
\end{gather*}
The matrix \( A_1 \) gives the family \( \lr_{5,\lambda(3)} \) with
restrictions on \( \lambda_i \) following from non-vanishing of \( a_1
= \lambda_1\lambda_2\lambda_3 \), \( a_2 = \prod_i
(1+\lambda_i)\prod_{i<j}(\lambda_i + \lambda_j) \), \( a_3 =
(\lambda_1+\lambda_2+\lambda_3)\prod_{i<j}(1 + \lambda_j + \lambda_k)
\).  The form \( A_2 \) corresponds to the family \(
\lr_{5,\lambda(2)} \) and the determinants \( a_1 =
\lambda_1\lambda_2^2 \), \( a_2 = 2\lambda_2 (1+\lambda_1)
(1+\lambda_2)^2 (\lambda_1+\lambda_2)^2 \), and \( a_3 =
(1+\lambda_1+\lambda_2)^2 (1+2\lambda_2) (\lambda_1+2\lambda_2) \)
should be non-zero.  From \( A_3 \) we obtain the family \(
\lr_{5(2),\lambda} \) with parameter value constrained by \( a_1 =
\lambda^2 \), \( a_2 = 4\lambda(1+\lambda)^4 \), \( a_3 =
(1+2\lambda)^2(2+\lambda)^2 \) being non-zero.  The matrix \( A_4 \)
corresponds to \( \lr_{5(1),\lambda} \) with \( \lambda \) constrained
by \( a_i\ne0 \) for \( a_1 = \lambda^3 \), \( a_2 =
8\lambda^3(1+\lambda)^3 \), \( a_3 = 3\lambda(1+2\lambda)^3 \).  The
algebra \( \lr_5 \) is from \( A_5 \).  Members of the \( \lr' \)- and
\( \lr'' \)-series occur when \( \ad_A \) has \( 2 \) or \( 4 \)
complex eigenvalues.  The algebra \( \lr'_{5,\lambda(3)} \)
corresponds to \( A_6 \); the conditions are that \( a_1 =
\lambda_1\lambda_2(1 + \lambda_3^2) \), \( a_2 = 2\lambda_3(\lambda_1
+ \lambda_2)(1 + (\lambda_1 + \lambda_3)^2)(1 + (\lambda_2 +
\lambda_3)^2) \) and \( a_3 = (\lambda_1 + 2\lambda_3)(\lambda_2 +
2\lambda_3)(1 + (\lambda_1 + \lambda_2 + \lambda_3)^2) \) are
non-zero.  The form \( A_7 \) gives the family \( \lr'_{5,\lambda(2)}
\) with \( \lambda_1 \), \( \lambda_2 \) constrained by \( a_1 =
\lambda_1^2(1+\lambda_2^2) \), \( a_2 =
4\lambda_1\lambda_2(1+(\lambda_1+\lambda_2)^2)^2 \), \( a_3 =
(\lambda_1+2\lambda_2)^2(1+(2\lambda_1+\lambda_2)^2) \) being
non-zero.  The matrix \( A_8 \) has \( \lambda_3\ne0 \) and
corresponds to the family \( \lr''_{5,\lambda(3)} \); restrictions on
parameter values follow from non-zero values for \( a_1 =
(1+\lambda_1^2) (\lambda_2^2+\lambda_3^2) \), \( a_2 =
4\lambda_1\lambda_2 ((\lambda_1+\lambda_2)^2 + (1+\lambda_3)^2)
((\lambda_1+\lambda_2)^2 + (1-\lambda_3)^2) \) and \( a_3 =
(\lambda_3^2 + (2\lambda_1+\lambda_2)^2) (1 +
(\lambda_1+2\lambda_2)^2) \).  Finally \( A_9 \) gives the algebra \(
\lr''_{5,\lambda} \).  The determinants \( a_1 = (1+\lambda^2)^2 \),
\( a_2 = 64\lambda^4(1+\lambda^2) \) and \( a_3 = (1+9\lambda^2)^2 \)
must be non-zero.

To analyse the cases \( (0^3,21) \) and \( (0^2,21,31) \) we follow
and modify arguments given in \cite{Mubarakzjanov:lie5}.  First
consider \( \lk \cong (0^3,21) \) which has \( H^1(\lk) \cong
\Span{e^1,e^2,e^3} \), \( H^2(\lk) \cong
\Span{e^{13},e^{14},e^{23},e^{24}} \) and \( H^3 \cong
\Span{e^{124},e^{134},e^{234}} \).  Write \( A(e_i) = a^k_ie_k \) for
\( i = 1,2,3,4 \).  From the relations \( \ad_A(e_4) =
[\ad_A(e_1),e_2] + [e_1,\ad_A(e_2)] \), \( 0 = \ad_A[e_i,e_3] =
[\ad_A(e_i),e_3] + [e_i,\ad_A(e_3)] \), \( i=1,2 \), we deduce that \(
a^4_4 = a^1_1+a^2_2 \) and \( a^1_4 = 0 = a^2_4 = a^3_4 = a^2_3 =
a^1_3 \).  After the transformation \( A\to A-a^4_2e_1+a^4_1e_2 \) we
may assume \( a^4_1 = a^4_2 = 0 \).  The restriction \( B = (b^k_i) \)
of \( \ad_A \) to the subspace \( \Span{e_1,e_2,e_3} \) has \( b^1_3 =
0 = b^2_3 \).  This may be transformed to Jordan form via \( e_1\to
ae_1+be_2+ce_3 \), \( e_2\to pe_1+qe_2+re_3 \), \( e_3\to se_3 \) with
\( aq-bp\ne0 \) and \( s\ne0 \).  Excluding degenerate matrices, we
may therefore take \( B = B_i \) to be one of: \(
\begin{smallpmatrix}
  1&0&0\\0&\lambda_1&0\\0&0&\lambda_2
\end{smallpmatrix}
\), \(
\begin{smallpmatrix}
  1&0&0\\1&1&0\\0&1&1
\end{smallpmatrix}
\), \(
\begin{smallpmatrix}
  1&0&0\\1&1&0\\0&0&\lambda
\end{smallpmatrix}
\), \(
\begin{smallpmatrix}
  \lambda_1&1&0&\\-1&\lambda_1&0\\0&0&\lambda_2
\end{smallpmatrix}
\), \(
\begin{smallpmatrix}
  \lambda&0&0&\\0&1&0\\0&1&1
\end{smallpmatrix}
\).  Consider first the case \( B_1 = \diag(1,\lambda_1,\lambda_2) \).
If \( \lambda_2\ne1+\lambda_1 \) we may assume, making a change \(
e_3\to e_3+\alpha e_4 \) if necessary, that \( a^4_3 = 0 \).  This
gives us the family \( \ld_{5,\lambda(2)} \).  The determinants \( a_1
= \lambda_1\lambda_2 \), \( a_2 =
(1+\lambda_2)(2+\lambda_1)(\lambda_1+\lambda_2)(1+2\lambda_1) \), and
\( a_3 = 2(1+\lambda_1)(2+\lambda_2+\lambda_1)(1+2\lambda_1+\lambda_2)
\) must be non-zero.  Turning next to the case \( \lambda_2 =
1+\lambda_1 \), let us assume \( a_3^4\ne0 \), otherwise we get a
member of the family \( \ld_{5,\lambda(2)} \).  After rescaling \(
e_i\to \abs{a^4_3}^{1/2}e_i \), for \( i = 1,2 \), \( e_4\to
\abs{a^4_3}e_4 \), we obtain the families \( \ld^{\pm}_{5(1),\lambda}
\) given by \( (0,21,\lambda.31,(1 + \lambda).41, (1 +
\lambda).51+32\pm41) \).  Scaling \( (e_1,e_4,e_5) \) by factors \(
(\lambda,\lambda^{-1},-1) \) and interchanging \( e_2 \) and \( e_3 \)
gives \( \ld^+_{5(1),\lambda} \cong \ld^-_{5(1),1/\lambda} \), so
there is only one family \( \ld_{5(1),\lambda}: = \ld^+_{5(1),\lambda}
\).  Restrictions on \( \lambda \) follow from non-vanishing of \( a_1
= \lambda(1+\lambda) \), \( a_2 = (2+\lambda)^2(1+2\lambda)^2 \) and
\( a_3 = 2(1+\lambda)(3+2\lambda)(2+3\lambda) \).  For the matrix \(
B_2 \) we may assume that \( a^4_3 = 0 \), so that we have the \(
(2,3) \)-trivial algebra \( \ld_{5} \).  The algebra \(
\ld_{5(2),\lambda} \) corresponds to the matrix \( B_3 \) with \(
a^4_3 = 0 \).  The following determinants \( a_1 = \lambda \), \( a_2
= 9(1+\lambda)^2 \), and \( a_3 = 4(3+\lambda)^2 \) must be non-zero.
For \( B_3 \) with \( a_3^4\ne0 \) we obtain the algebra \(
\ld_{5(2)}^{\pm} \); this requires a rescaling \( e_i\to
\abs{a^4_3}^{1/2}e_i \), for \( i = 1,2 \), \( e_4\to \abs{a^4_3}e_4
\).  From \( B_4 \) we obtain \( \ld'_{5,\lambda(2)} \) when \( a_3^4
= 0 \).  The requirement that \( a_1 = \lambda_2(1+\lambda_1^2) \), \(
a_2 = (1+9\lambda_1^2)(1+(\lambda_1+\lambda_2)^2) \), \( a_3 =
4\lambda_1(1+(3\lambda_1+\lambda_2)^2) \) are non-zero enforces
restrictions on the \( \lambda_i \)'s.  When \( a_3^4\ne0 \) we find,
after appropriate rescaling, that \( B_4 \) corresponds to the family
\( \ld'^{\pm}_{5,\lambda} \).  The determinants \( a_1 =
2\lambda(1+\lambda^2) \), \( a_2 = (1+9\lambda^2)^2 \) and \( a_3 =
4\lambda(1+25\lambda^2) \) must be non-zero.  For the matrix \( B_5 \)
we must have \( a_3^4 = 0 \) and so we get the family \(
\ld_{5(3),\lambda} \).  The allowed values for \( \lambda \) are
deduced from the determinants \( a_1 = \lambda \), \( a_2 =
2(1+\lambda)(1+2\lambda)(2+\lambda) \), \( a_3 =
4(1+\lambda)^2(3+\lambda) \) being non-zero.
    
Finally, for \( \lk\cong(0^2,21,31) \) we have \( H^1(\lk) =
\Span{e^1,e^2} \), \( H^2(\lk)\cong\Span{e^{14},e^{23}} \), \(
H^3(\lk)\cong\Span{e^{134},e^{234}} \).  As above, write \( A(e_i) =
a^k_ie_k \).  Considering the relations \( 0 = \ad_A[e_2,e_3] =
[\ad_A(e_2),e_3] + [e_2,\ad_A(e_3)] \), \( \ad_A(e_3) =
[\ad_A(e_1),e_2] + [e_1,\ad_A(e_2)] \) and \( \ad_A(e_4) =
[\ad_A(e_1),e_3] + [e_1,\ad_A(e_3)] \), one finds \( a^1_2 = a^1_3 =
a^1_4 = a^2_3 = a^2_4 = a^3_4 = 0 \), \( a^3_2 = a^4_3 \) and \( a^4_4
= a^1_1 + a^3_3 \), \( a^3_3 = a^1_1 + a^2_2 \).  After the
transformation \( A\to A-a^3_2e_1+a^3_1e_2+a^4_1e_3 \) we may assume
that \( \ad_A \) takes the form \( \diag(p,q,p+q,2p+q) + A' \), where
\( A' \) only has non-zero entries \( {a'}^2_1 = a^2_1 \) and \(
{a'}^4_2 = a^4_2 \), below the diagonal.  We then obtain \(
\lp_{5,\lambda} \) and \( \lp_5 \) as follows.  As \( \lk = \g' \) one
has \( p \ne 0 \) and we may to rescale \( \ad_A \) by \( 1/p \).  If
\( q\ne p \) we may consider the transformation \( e_1\to e_1+ {a^2_1
  e_2}/(p-q) \).  After appropriate transformations \( e_1\to e_1+ae_4
\), \( e_2\to e_2+be_4 \), we obtain the algebra \( \lp_{5,\lambda} \)
with \( \lambda = q/p \).  For this family we have determinants \( a_1
= \lambda \), \( a_2 = (1+2\lambda)(3+\lambda) \) and \( a_3 =
6(1+\lambda)(2+\lambda) \), so that \( a_i\ne0 \) enforces \( \lambda
\) to be as specified in Table~\ref{tab:5sp}.  Consider now \( q = p
\) and note that we may assume \( a^2_1\ne0 \), since otherwise we end
with~\( \lp_{5,\lambda} \).  Making a change \( e_i\to a^2_1 e_i \)
for \( i = 2,3,4 \) and then transforming \( e_2\to e_2+c e_4 \), we
get the algebra \( \lp_5 \).

This concludes the proof of Proposition \ref{prop:solb345}.

\paragraph{Unimodular}

The lists of \((2,3)\)-trivial algebras in dimensions up to and
including five reveal that algebraic properties of this class are not
fully reflected in low-dimensional examples. A Lie algebra \( \g \) is
called \emph{unimodular} if the homomorphism \( \chi\colon \g\to\bR \)
given by \( \chi(x) = \trace(\ad(x)) \) has trivial image. By direct
inspection we observe that the \( (2,3) \)-trivial Lie algebras of
dimensions two, three and four are not unimodular. On the other hand
there are infinitely many five-dimensional algebras with this
property:

\begin{corollary}
  The unimodular \( (2,3) \)-trivial Lie algebras of dimension up to
  and including five are \( \bR \), \( \lr_{5(1),-1/3} \), \(
  \lr_{5,\lambda,-(1+\lambda)/2} \), \(
  \lr_{5,\lambda,\mu,-(1+\lambda+\mu)} \), \(
  \lr'_{5,\lambda,-\lambda} \), \( \lr''_{5,\lambda,-\lambda,\mu} \),
  \( \lr'_{5,\lambda,\mu,-(\lambda+\mu)/2} \), \( \ld_{5(2),-4} \), \(
  \ld_{5,\lambda,-2(1+\lambda)} \), \( \ld_{5(3),-3/2} \), \(
  \ld'_{5,\lambda,-4\lambda} \) and \( \lp_{5,-4/3} \), where
  parameters satisfy the conditions in Table~\ref{tab:5sp}.
\end{corollary}

\paragraph{Higher dimensions}

The quest for higher-dimensional examples is easily met.  Indeed, one
may construct infinite families of \( (2,3) \)-trivial Lie algebras
following the methods invoked in the proof of Proposition
\ref{prop:solb345}.  In fact all the families appearing in dimension
five have higher dimensional generalisations. Let us show how to
obtain the following examples:
\begin{asparaitem}
\item \( \lr_n \): \( (0,21{+}31,31{+}41,\dots,(n{-}1)1{+}n1,n1) \),
\item \( \lr_{n(k-1),\lambda} \): \( (0, 21{+}31, \dots,
  (k{-}1)1+k1,k1, \lambda.(k{+}1)1+(k{+}2)1, \dots, \lambda.(n{-}1)1 +
  n1,\allowbreak \lambda.n1) \), with \( k > 2 \) and \( \lambda \ne
  0,{-}1,-2,-1/2 \),
\item \( \lr_{n,\lambda(k)} \): \(
  (0,21,\lambda_1.31,\dots,\lambda_{k{-}1}.(k{+}1)1,\lambda_k.(k{+}2)1
  + (k{+}3)1,\dots,\lambda_k.(n{-}1)1 + n1,\allowbreak \lambda_k.n1)
  \), with \(n > k + 2 \) and non-zero \( \lambda_i \), \(
  1{+}\lambda_i \), \( \lambda_i{+}\lambda_j \), \( 1{+}2\lambda_k \),
  \( \lambda_i{+}2\lambda_k \), \( 1{+}\lambda_i{+}\lambda_j \) \(
  (i<j) \) and \( \lambda_i{+}\lambda_j{+}\lambda_{\ell} \) \(
  (i<j<\ell) \),
\item \( \ld_{n,\lambda(n-3)} \): \( (0, 21,\lambda_1.31, \dots,
  \lambda_{n{-}3}.(n{-}1)1, (1{+}\lambda_1).n1{+}32) \), with
  \(\lambda_i \ne 0,{-}1\) for all \(i\), \(\lambda_1\ne {-2}, {-}1/2,
  {-}\lambda_i, {-}1/2(1+\lambda_i), {-}2{-}\lambda_i,
  {-}\lambda_i{-}\lambda_j\) for \(1<i\), \(1<i<j\) and non-zero \(
  \lambda_i+\lambda_j \), \( 1{+}\lambda_i+\lambda_j \) \((1<i<j)\),
  \( \lambda_i{+}\lambda_j{+}\lambda_k \) \( (1<i<j<k) \).
\end{asparaitem}

The members of the \( \lr \)-series have \( \lk\cong\bR^{n-1} \) and
\( \ad_A \) belongs to the list \( J(n-1,1) \), \( J(k-1,1) \oplus
J(n-k,\lambda) \), \( \diag(1,\lambda_1,\dots,\lambda_{k-1}) \oplus
J(n-k-1,\lambda_k) \), where \( J(m,a) \) is an \( m\times m \)-Jordan
block with \( a \) on the diagonal and \( 1 \) immediately above the
diagonal.  The first matrix corresponds to \( \lr_n \), the second
corresponds to the families \( \lr_{n(k-1),\lambda} \) and the third
one gives \( \lr_{n,\lambda(k)} \).  For the latter two cases the
requirement that \( A \) must act invertibly on cohomology enforces
some restrictions on parameters.  As \( A \) acts on \(
H^1(\bR^{n-1})\cong\bR^{n-1} \) by a lower triangular matrix, these
restrictions are easily determined: the sum of one, two or three
diagonal elements must be non-zero.

The family \( \ld_{n,\lambda(n-3)} \) has \( \lk\cong(0^{n-2},21) \)
and \( \ad_A \) is \( \diag(1,\lambda_1, \dots, \lambda_{n-3},
1+\lambda_1) \).  Now \( A \) acts diagonally on \( \lk^* \), and
restrictions on parameters may therefore be read off directly from the
cohomology groups \( H^1(\lk) \cong \lk^* \ominus \Span{e^{n-1}}\),
\(H^2(\lk) \cong \Lambda^2\lk^* \ominus \Span{e^{12} ,
  e^{i(n-1)}\colon\, i>2}\), \(H^3(\lk) \cong \Lambda^3\lk^* \ominus
\Span{e^{12i},e^{jk(n-1)}\colon\, 2<i<n-1,2<j<k}\).

An alternative way of constructing infinite families of \( (2,3)
\)-trivial algebras goes via positive gradings of infinite families:
\begin{asparaitem}
\item \( \lf^1_n \): \( (0, 21, 31, 2.41 + 32, 3.51 + 42, \dots, (n -
  2).n1 + (n - 1)2) \),
\item \( \lf^2_n \): \( (0, 21, 2.31, 3.41 + 32, 4.51 + 42, 5.61 + 52
  + 43, \dots, (n - 1).n1+(n - 1)2+{(n - 2)3})\),
\item \( \lf^3_n \): \( (0, 21, 31, 2.41 + 32, \dots, (n - 3).(n - 1)1
  + (n - 2)2, (n - 2).n1 + (n - 1)2 - {(n - 1)3} +(n - 2)3 -
  \dotsb)\).
\end{asparaitem}
Here \( ({\lf^1_n})'=(0^2,21,\dots,(n-2)1) \) has positive grading \(
1^22\dotsb (n-2) \).  The derived algebra \( ({\lf^2_n})' =
(0^2,21,31,41+32,\dots,(n-2)1+(n-3)2) \) admits grading \( 12\dotsb
(n-1) \) and \( ({\lf^3_n})' = (0^2, 21, 31, \dots, (n-3)1,
(n-2)1-(n-2)2+(n-3)3 -\dots -(-1)^k(k+1)k) \) with \( n=2k+1 \) has
positive grading \( 1^223\dotsb (n-2) \).

\medbreak In conclusion, the above exposition shows that the class of
\((2,3)\)-trivial algebras is quite rich. Yet, because of
Theorem~\ref{thm:2-3-trivial}, this class is much more accessible than
the larger class of solvable algebras.

\paragraph*{Acknowledgements}
We gratefully acknowledge financial support from \textsc{ctqm} and
\textsc{geomaps}.  AFS is also partially supported by the Ministry of
Science and Innovation, Spain, under Project \textsc{mtm}{\small
  2008-01386} and thanks \textsc{nordita} for hospitality.

% \bibliography{papers} \bibliographystyle{plainnat}

\bigskip

\begin{small}
  \parindent0pt\parskip\baselineskip

  T.B.Madsen \& A.F.Swann

  Department of Mathematics and Computer Science, University of
  Southern Denmark, Campusvej 55, DK-5230 Odense M, Denmark

  \textit{and}

  CP\textsuperscript3-Origins, Centre of Excellence for Particle
  Physics Phenomenology, University of Southern Denmark, Campusvej 55,
  DK-5230 Odense M, Denmark.

  \textit{E-mail}: \url{tbmadsen@imada.sdu.dk},
  \url{swann@imada.sdu.dk}
\end{small}

\end{document}